\documentclass{article}
\usepackage{latexsym,amsthm,amssymb,amscd,amsmath}
\newtheorem{thm}{Theorem}[section]
\newtheorem{cor}[thm]{Corollary}

\newtheorem{lem}[thm]{Lemma}

\newtheorem{rem}[thm]{Remark}

\newcommand{\be}{\begin{equation}}
\newcommand{\ee}{\end{equation}}
\newcommand{\ben}{\begin{enumerate}}
\newcommand{\een}{\end{enumerate}}
\newcommand{\beq}{\begin{eqnarray}}
\newcommand{\eeq}{\end{eqnarray}}
\newcommand{\beqn}{\begin{eqnarray*}}
\newcommand{\eeqn}{\end{eqnarray*}}

\title{On Cartan Spaces with $m$-th Root Metrics}
\author{A. Tayebi, A. Nankali and E. Peyghan}
\begin{document}

\maketitle
\begin{abstract}In this paper, we define some non-Riemannian curvature  properties for  Cartan spaces. We consider Cartan space with the $m$-th root metric. We prove that every $m$-th root Cartan space of isotropic Landsberg curvature, or isotropic mean Landsberg curvature, or isotropic mean Berwald curvature reduces to a  Landsberg, weakly Landsberg and weakly Berwald space, respectively. Then we show that   $m$-th root Cartan space of  almost vanishing $\bf H$-curvature satisfies ${\bf H}=0$.
\end{abstract}

\section{Introduction}
\'{E}. Cartan  has originally introduced a Cartan space, which is considered as dual of Finsler space \cite{Car}. Then  Rund and Brickell studied the relation between these two spaces \cite{Bric}\cite{Rund}. The theory of Hamilton
spaces was introduced by Miron \cite{Mir2}.
He proved that Cartan space is a particular
case of Hamilton space.

Let us denote the Hamiltonian structure on a manifold $M$ by $(M, H(x,p))$. If the fundamental function $H(x,p)$ is 2-homogeneous on the fibres of the cotangent bundle $T^*M$, then the
notion of Cartan space is obtained \cite{MirHri}\cite{PAT}\cite{PT3}. Indeed, the modern formulation of the notion of Cartan spaces is due of the Miron \cite{Mir2}\cite{Mir3}. Based on the studies of   Kawaguchi \cite{Kaw},
Miron \cite{MirHri},   Hrimiuc-Shimada \cite{HriShi},   Anastasiei-Antonelli \cite{Anas}, Maz\.{e}tis \cite{Maz1}\cite{Maz2}\cite{Maz3}, Urbonas \cite{U} etc., the geometry of Cartan spaces is today an important chapter of differential geometry.

Under  Legendre transformation, the Cartan spaces appear as the dual of the Finsler spaces \cite{Mir2}.  Finsler geometry was developed since 1918 by Finsler, Cartan, Berwald, Akbar-Zadeh, Matsumoto,  Shen and many others, see  \cite{AZ}\cite{ShDiff}. Using this duality several important results in the Cartan spaces can be obtained: the canonical nonlinear
connection, the canonical metrical connection, the notion of $(\alpha, \beta)$-metrics, the theory of $m$-root metrics,  etc \cite{Maz1}\cite{Maz2}\cite{Na}. Therefore, the
theory of Cartan spaces has the same symmetry and beauty like
Finsler geometry. Moreover, it gives a geometrical framework for
the Hamiltonian theory of Mechanics or Physical fields.

The theory of $m$-th root metric has been developed by H. Shimada \cite{Shim}, and applied to Biology as an ecological metric. It is regarded as a direct generalization of Riemannian metric in a sense, i.e., the second root metric is a Riemannian metric.  Recently studies,  shows that the theory of  $m$-th root Finsler metrics play a very important role in physics, theory of space-time structure, general relativity and seismic ray theory \cite{TN1}\cite{TN2}\cite{TPS}.

An $n$-dimensional Cartan space $C^n$ with $m$-th root metric is a Cartan structure $C^n=(M^n,K(x,p))$ on differentiable
$n$-manifold $M^n$ equipped with the fundamental function  $K(x,p)=\sqrt[m]{a^{i_1i_2...i_m}(x)p_{i_1}p_{i_2}...p_{i_m}}$ where $a^{i_1i_2...i_m}(x)$, depending on the position alone, is
symmetric in all the indices $i_1,i_2,...,i_m$ and $m\geq3$.  The Hessian of $K^2$ give us the fundamental tensor $\textbf{g}$ of Cartan space.  Taking a vertical derivation of ${\bf g}$ give us the Cartan torsion ${\bf C}$. The rate of change of the Cartan torsion along geodesics, $\textbf{L}$ is said to be Landsberg curvature. A Cartan metric with  ${\bf L}=cF{\bf C}$ is called  isotropic Landsberg metric, where
$c=c(x)$ is a scalar function on $M$. In this paper, we  prove that
every $m$-root Cartan spaces with isotropic  Landsberg curvature is a  Landsberg space.

\begin{thm}\label{mainthm0}
Let $(M,K)$ be an $m$-th root Cartan space. Suppose that $K$ is
isotropic Landsberg metric, ${\bf L}+cK{\bf C}=0$ for some scalar
function $c=c(x)$ on $M$. Then $K$ reduces to a  Landsberg metric.
\end{thm}

Taking a trace of Cartan torsion ${\bf C}_y$ and Landsberg curvature ${\bf L}_y$ give us the mean Cartan
torsion ${\bf I}_y$ and mean Landsberg curvature ${\bf J}_y$, respectively. A Cartan metric with  ${\bf J}=0$ and ${\bf J}=cF{\bf I}$ is called weakly Landsberg and isotropic mean Landsberg metric, respectively, where
$c=c(x)$ is a scalar function on $M$. We show that
every $m$-root Cartan spaces of isotropic mean Landsberg curvature reduces to  weakly Landsberg space.

\begin{thm}\label{mainthm1}
Let $(M,K)$ be an $m$-th root Cartan space. Suppose that $K$ has
isotropic mean Landsberg curvature, ${\bf J}+cK{\bf I}=0$ for some
scalar function $c=c(x)$ on $M$. Then $K$ reduces to a weakly
Landsberg metric.
\end{thm}
Taking a trace of Berwald curvature of  Cartan metric $K$  gives rise the
$E$-curvature.  The Cartan metric $K$ with ${\bf E}=0$ and ${\bf E}=\frac{n+1}{2}cK{\bf h}$ is called weakly Berwald
 and  isotropic mean Berwald metric, respectively, where $c=c(x)$ is a
scalar function on $M$ and ${\bf h}=h^{ij}dx_idx_j$ is the angular
metric.
\begin{thm}\label{mainthm3}
Let $(M,K)$ be an m-th root Cartan space. Suppose
that $K$ has isotropic mean Berwald curvature ${\bf E}=\frac{n+1}{2} c K{\bf h}$, for some scalar function $c=c(x)$  on $M$. Then $K$ reduces to a weakly Berwald metric.
\end{thm}

Akbar-Zadeh introduces the non-Riemannian quantity $\bf H$ which is
obtained from the mean Berwald curvature bye the covariant
horizontal differentiation along geodesics. More precisely, the
non-Riemannian quantity ${\bf H} =H^{ij}dx_i\otimes dx_j$ is defined
by $H^{ij}:=E^{ij|s}p_s$. The Cartan metric $K$ is called  of almost
vanishing $\bf H$-curvature if $H^{ij}=\frac{n+1}{2K}\theta h^{ij}$, where $\theta$ is a 1-form on $M$.
\begin{thm}\label{mainthm2}
Let $(M,K)$ be an $n$-dimensional $m$-th root Cartan space. Suppose
that $K$ has almost vanishing $\bf H$-curvature, ${\bf H}=\frac{n+1}{2}K^{-1}\theta {\bf h}$  for some 1-form $\theta$ on $M$. Then $\bf H=0$.
\end{thm}

\section{Preliminaries}\label{sectionP}

A Cartan spaces is a pair
$C^n=(M^n,K(x,p))$ such that the following axioms hold good:\\
1. $K$ is a real positive function on the cotangent bundle $T^*M$,
differentiable on $T^*M_0:= T^*M\setminus\{0\}$ and continuous on the null
section of the canonical projection
\[
\pi^* :T^*M\rightarrow M ;
\]
2. $K$ is positively 1-homogenous with respect to the momenta
$p_i$;\\
3. The Hessian of $K^2$, with the elements
\[
g^{ij}(x,p)=\frac{1}{2}\frac{\partial^2K^2}{\partial p_i\partial
p_j}
\]
is positive-defined on $T^*M_0$.

An $n$-dimensional Cartan space $C^n$ with $m$-th root metric is by
definition a Cartan structure $C^n=(M^n,K(x,p))$ on differentiable
$n$-manifold $M^n$ equipped with the fundamental function $K(x,p)$
such that
\[
K(x,p)=\sqrt[m]{a^{i_1i_2...i_m}(x)p_{i_1}p_{i_2}...p_{i_m}}
\]
where $a^{i_1i_2...i_m}(x)$, depending on the position alone, is
symmetric in all the indices $i_1,i_2,...,i_m$ and $m\geq3$.

From $K(x,y)$ we define Cartan symmetric tensors of order $r\ (1\leq
r\leq m-1)$ with the components
\[
a^{i_1i_2...i_r}(x,p)=\frac{1}{K^{m-r}}a^{i_1i_2...i_rj_1j_2...j_{m-r}}p_{j_1}p_{j_2}...p_{j_{m-r}}
\]
Thus we have
\begin{eqnarray*}
&&a^i=[a^{ii_2i_3...i_m}(x)p_{i_2}p_{i_3}...p_{i_{m}}]\diagup K^{m-1},\\
&&a^{ij}=[a^{iji_3i_4...i_m}(x)p_{i_3}p_{i_4}...p_{i_{m}}]\diagup
K^{m-2},\\
&&a^{ijk}=[a^{ijki_4i_5...i_m}(x)p_{i_4}p_{i_5}...p_{i_{m}}]\diagup
K^{m-1}.
\end{eqnarray*}
The normalized supporting element is given by
\[
l^i=\dot{\partial}^iK,
\]
where $\dot{\partial}^i=\frac{\partial}{\partial p_i}$. The fundamental metrical d-tensor is
\[
g^{ij}=\frac{1}{2}\dot{\partial}^i\dot{\partial}^j K^2
\]
and the angular metrical d-tensor is given by
\[
h^{ij}=K\dot{\partial}^i\dot{\partial}^j K.
\]
The following hold
\begin{eqnarray*}
&&l^i=a^i,\\
&&g^{ij}=(m-1)a^{ij}-(m-2)a^ia^j,\\
&&h^{ij}=(m-1)(a^{ij}-a^ia^j).
\end{eqnarray*}
From the positively 1-homogeneity of the $m$-th root Cartan metrical
function, it follows that
\[
K^2(x,p)=g^{ij}(x,p)p_ip_j=a^{ij}(x,p)p_ip_j .
\]
Since $det(g^{ij})=(m-1)^{n-1}det(a^{ij})$, the regularity of the
$m$-th metric is equivalent to $det(a^{ij})\neq0$. Let us suppose now
that the d-tensor $a^{ij}$ is regular, that is there exists the
inverse matrix $(a^{ij})^{-1}=(a_{ij})$. Obviously, we have
\[
a_i.a^i=1,
\]
where
\[
a_i=a_{is}a^s=\frac{p_i}{K}.
\]
Under these assumptions, we obtain the inverse components
$g_{ij}(x,p)$ of the fundamental metrical d-tensor $g^{ij}(x,p)$,
which are given by
\begin{eqnarray}\
g_{ij}=\frac{1}{m-1}a_{ij}+\frac{m-2}{m-1}a_ia_j.
\end{eqnarray}
We have
\begin{eqnarray*}
&&\dot{\partial}^k(a^{ij})=\frac{(m-2)}{K}[a^{ijk}-a^{ij}a^k],\\
&&\dot{\partial}^k(a^{i})=\frac{(m-1)}{K}[a^{ik}-a^{i}a^k],\\
&&\dot{\partial}^k(a^{i}a^j)=\frac{(m-1)}{K}[a^{ik}a^j+a^{jk}a^i-2a^{i}a^ja^k].
\end{eqnarray*}
The Cartan tensor $C^{ijk}=-\frac{1}{2}(\dot{\partial}^kg^{ij})$ are
given in the form
\begin{eqnarray}
C^{ijk}=-\frac{(m-1)(m-2)}{2K}(a^{ijk}-a^{ij}a^k-a^{jk}a^i-a^{ki}a^j+2a^ia^ja^k).\label{m0}
\end{eqnarray}
The $m$-th Christoffel symbols is defined by
\begin{eqnarray}
\nonumber\{i_1...i_m,j\}=\frac{1}{2(m-1)}(\partial^{i_1}a^{i_2...i_mj}\!\!\!\!&+&\!\!\!\!\!\partial^{i_2}a^{i_3...i_mi_1j}\\
\!\!\!\!&+&\!\!\!\!\! \cdots +\partial^{i_m}a^{i_1...i_{m-1}j}-\partial^{j}a^{i_1...i_m}),
\end{eqnarray}
where the cyclic permutation is applied to $(i_1...i_m)$ in the
first $m$ terms of the right-hand side.

Now, if we write the equations of
geodesics in the usual form
\begin{eqnarray}
\frac{d^2x_i}{ds^2}+2G_i(x,\frac{dx}{ds})=0 ,
\end{eqnarray} then
the quantities $G_i(x,y)$ are  given by
\begin{eqnarray}
a^{hr}G_r=\frac{1}{mK^{m-2}}\{00...0,h\} ,\label{m3}
\end{eqnarray}
where we denote by the index $0$ the multiplying  by $p_i$ as usual,
that is
\[
\{00...0,h\}=\{i_1i_2...i_m,h\}p_{i_1}p_{i_2}...p_{i_m}   ,
a^{hr}=a^{hr00...0}/K^{m-2}.
\]
Using the definition of $a^{hr}$, we can write (\ref{m3}) in
the form
\be
a^{hr00...0}G_r=\frac{1}{m}\{00...0,h\}.\label{C1}
\ee
Differentiating of (\ref{C1}) with respect to $p_i$ yields
\be
a^{hr00...0}G^i_r+(m-2)a^{hri00...0}G^r=\{i00...0,h\},\label{C2}
\ee
where $G^i_r=\dot{\partial}^iG_r$. By differentiating of (\ref{C2}) with respect to $p_j$, we have
\begin{eqnarray*}
a^{hr00...0}G^{ij}_r+(m-2)\big[a^{hrj00...0}G^i_r+a^{hri00...0}G^j_r\big]\!\!\!\!&+&\!\!\!\!\! (m-2)(m-3)a^{hrij00...0}G_r\\
\!\!\!\!&=&\!\!\!\!\!(m-1)\{ij00...0,h\},
\end{eqnarray*}
where $G_r^{ij}=\dot{\partial}^jG^i_r$ constitute the coefficients of the
Berwald connection $B\Gamma=(G^{ij}_r,G^i_r)$. The above  equations can be written in the following forms
\begin{eqnarray}
K^{m-3}\big[Ka^{hr}G^i_r+(m-2)a^{hri}G_r\big]=\{i00...0,h\},
\end{eqnarray}
and
\begin{eqnarray}
\nonumber K^{m-3}\Big[Ka^{hr}G^{ij}_r+(m-2)(a^{hrj}G^i_r+a^{hri}G^j_r)\Big]\!\!\!\!&+&\!\!\!\!\!K^{m-4}(m-2)(m-3)a^{hrij}G_r\\
\!\!\!\!&=&\!\!\!\!\!(m-1)\{ij00...0,h\}.\label{C3}
\end{eqnarray}
By differentiation of (\ref{C3}) with respect to $p_k$, we get the Berwald curvature  of Berwald connection as follows
\begin{eqnarray}
\nonumber K^{m-3}\Big[\!\!\!\!\!\!\!&&\!\!\!\!\!\!\!\!Ka^{hr}G^{ijk}_r+(m-2)[a^{hir}G^{jk}_r+(i,j,k)]\Big]\\
\nonumber \!\!\!\!&+&\!\!\!\!\!(m-2)(m-3)K^{m-5}\Big[K[a^{hijr}G^k_r+(i,j,k)]+(m-4)a^{hijkr}G_r\Big]\\
\!\!\!\!&=&\!\!\!\!\!(m-1)(m-2)\{ijk00...0,h\},\label{m4}
\end{eqnarray}
where $\{... , (ijk)\}$ shows the cyclic permutation of the indices
$i,j,k$ and summation. Multiplying  (\ref{m4}) with $p_h$ yields
\begin{eqnarray}
\nonumber K^{m-2}\Big[\!\!\!\!\!\!\!&&\!\!\!\!\!\!\!\!p^rG_r^{ijk}+(m-2)[a^{ir}G^{jk}_r+(i,j,k)]\Big]\\
\nonumber \!\!\!\!&+&\!\!\!\!\!(m-2)(m-3)K^{m-4}\Big[K[a^{ijr}G^k_r+(i, j, k)]+(m-4)a^{ijkr}G_r\Big]\\
\!\!\!\!&=&\!\!\!\!\!(m-1)(m-2)\{ijk00...0,0\}.\label{m5}
\end{eqnarray}
\begin{rem} \emph{In the equations (\ref{m4}) and (\ref{m5}), we have some
terms with coefficients $(m-3)$ and $(m-4)$. We shall be concerned
mainly with cubic metric $(m=3)$ and quartic  metric $(m=4)$
\[
K^3=a^{ijk}(x)p_ip_jp_k ~~,~~ K^4=a^{hijk}(x)p_hp_ip_jp_k.
\]
For these metrics, it is supposed that the terms
with $(m-3)$ and $(m-4)$ vanish, respectively. For instance,
(\ref{m4}) of a cubic metric is reduced to following}
\[
Ka^{hr} G^{ijk}_r+\{a^{hir}G^{jk}_r+(i,j,k)\}=\{ijk,h\}.
\]
\end{rem}

%------------------------------------------------------------------------------------------------------------
\section{ Proof of Theorem \ref{mainthm0} }
%------------------------------------------------------------------------------------------------------------
In this section, we are going to prove  Theorem \ref{mainthm0}.  First, we remark the following.

\begin{lem}{\rm (\cite{YY})}\label{lemG}
\emph{Let $(M,K)$  be an $m$-th root Cartan space. Then the spray
coefficients of $K$ are given by following}
\[
G_r=\frac{1}{m}\{00...0,h\}\diagup a^{hr00...0}.
\]
\end{lem}

\bigskip

Now, we can prove the Theorem \ref{mainthm0}.

\bigskip
\noindent {\it\bf Proof of Theorem \ref{mainthm0}}: By assumption, the Cartan metric $K$ has isotropic Landsberg curvature  ${\bf L}=cK{\bf C}$  where $c=c(x)$ is a scalar function on $M$. By definition, we have
\[
L^{ijk}=-\frac{1}{2}p^sG^{ijk}_s,
\]
where
\begin{eqnarray}
\nonumber p^s=g^{sj}p_j\!\!\!\!&=&\!\!\!\!\! [(m-1)a^{sj}-(m-2)a^sa^j]p_j\\
\nonumber\!\!\!\!&=&\!\!\!\!\!(m-1)Ka^s-(m-2)Ka^s\\
\nonumber \!\!\!\!&=&\!\!\!\!\! Ka^s\\
\!\!\!\!&=&\!\!\!\!\! a^{s00...0}\diagup K^{m-2}.
\end{eqnarray}
Then we get
\[
L^{ijk}=-\frac{1}{2}a^{s00...0}\diagup K^{m-2}G^{ijk}_s.
\]
By assumption, we have
\begin{eqnarray*}
-\frac{1}{2K^{m-2}}a^{s00...0}G^{ijk}_s=cKC^{ijk}
\end{eqnarray*}
or equivalently
\begin{eqnarray}
a^{s00...0}G^{ijk}_s=\frac{c}{K^{2-m}}(m-1)(m-2)(a^{ijk}-a^{ij}a^k-a^{jk}a^i-a^{ki}a^j+2a^ia^ja^k).\label{m6}
\end{eqnarray}
By Lemma \ref{lemG}, the left-hand side of (\ref{m6}) is rational
function, while its right-hand side is an irrational function. Thus,
either $c=0$ or $a$ satisfies the following
\begin{eqnarray}
a^{ijk}-a^{ij}a^k-a^{jk}a^i-a^{ki}a^j+2a^ia^ja^k=0.\label{m7}
\end{eqnarray}
Plugging (\ref{m7}) into (\ref{m0}) implies that $C^{ijk}=0$. Hence, $K$ is Riemannian metric, which contradicts with our
assumption. Therefore, $c=0$. This completes the proof.
\qed

%------------------------------------------------------------------------------------------------------------
\section{ Proof of Theorem \ref{mainthm1} }
%------------------------------------------------------------------------------------------------------------

The quotient ${\bf J}/{\bf I}$ is regarded as the relative rate of
change of mean Cartan torsion ${\bf I}$ along Cartan geodesics. Then
$K$ is said to be isotropic mean Landsberg metric if ${\bf J}=cK \bf
I$, where $c=c(x)$ is a scalar function on $M$. In this section, we
are going to prove  the Theorem \ref{mainthm1}. More precisely, we
show that every  $m$-th root isotropic mean Landsberg metric reduces
to a weakly Landsberg metric.

\bigskip

\noindent {\it\bf Proof of Theorem \ref{mainthm1}}: The mean Cartan
tensor of $K$ is given by following
\begin{eqnarray*}
I^i\!\!\!\!&=&\!\!\!\!\!g_{jk}C^{ijk}\\
\!\!\!\!&=&\!\!\!\!\!\frac{-(m-2)}{2K}\{a^{ij}_j-\delta^i_ka^k-na^i-\delta^i_ja^j+2a^i\}\\
\!\!\!\!&=&\!\!\!\!\!\frac{-(m-2)}{2K}\{a^{ir}_r-na^i\}.
\end{eqnarray*}
The mean Landsberg curvature of $K$ is given by
\begin{eqnarray*}
J^i\!\!\!\!&=&\!\!\!\!\!g_{jk}L^{ijk}\\
\!\!\!\!&=&\!\!\!\!\![\frac{1}{m-1}a_{jk}+\frac{m-2}{m-1}a_ja_k][-\frac{1}{2}a^{s00...0}\diagup
K^{m-2}G^{ijk}_s].\\
\end{eqnarray*}
Since ${\bf J}=cF \bf I$, then we have
\[
c(m-1)(m-2)\{a^{ijk}-a^{ij}a^k-a^{jk}a^i-a^{ki}a^j+2a^ia^ja^k\}=a^{s00...0}\diagup
K^{m-2}G^{ijk}_s.
\]
Thus we get
\be
a^{s00...0}G^{ijk}_s=c
K^{m-2}(m-1)(m-2)\{a^{ijk}-a^{ij}a^k-a^{jk}a^i-a^{ki}a^j+2a^ia^ja^k\}.\label{ml}
\ee
By  Lemma \ref{lemG}, the left hand side of (\ref{ml}) is a rational function with respect to $y$,
while its right-hand side is an irrational function with respect to $y$. Thus, either
$c=0$ or $a$ satisfies the following
\[
a^{ijk}-a^{ij}a^k-a^{jk}a^i-a^{ki}a^j+2a^ia^ja^k=0.
\]
That implies that $C^{ijk}=0$. Hence, $K$ is Riemannian metric, which contradicts with our assumption. Therefore,
$c=0$. This completes the proof. \qed

%------------------------------------------------------------------------------------------------------------
\section{ Proof of Theorem \ref{mainthm3} }
%------------------------------------------------------------------------------------------------------------
Let $(M, K)$ be a Cartan space of dimension $n$. Denote by
$\tau(x,y)$ the distortion of the Minkowski norm $K_{x}$ on
$T^*_{x}M_{0}$,  and $\sigma(t)$ be the geodesic with $\sigma(0)=x$
and $\dot{\sigma}(0)=y$. The rate of change of  $\tau(x,y)$ along
Cartan geodesics $\sigma(t)$ called $S$-curvature. $K$ is said to
have isotropic $S$-curvature if
\[
{\bf S}= (n+1) c K.
\]
where $c=c(x)$ is a scalar function on $M$. $K$ is called of  almost isotropic $S$-curvature if
\[
{\bf S}=(n+1)c K+\eta,
\]
where $c=c(x)$ is a scalar function and $\eta=\eta^i(x)p_i$ is a 1-form on $M$.

\begin{rem}
\emph{By taking twice vertical covariant derivatives  of the $S$-curvature, we get the $E$-curvature
  \[
E^{ij}(x,p):=\frac{1}{2}\frac{\partial^2{\bf S}}{\partial p_i\partial
p_j}.
\]
It is remarkable that, we can get the $E$-curvature by taking a trace of Berwald curvature of  Cartan metric $K$, also. The Cartan metric $K$ is called weakly Berwald
metric if ${\bf E}=0$ and is said to have isotropic mean Berwald
curvature if ${\bf E}=\frac{n+1}{2}cK{\bf h}$, where $c=c(x)$ is a
scalar function on $M$ and ${\bf h}=h^{ij}dx_idx_j$ is the angular
metric.}
\end{rem}

In this section, we are going to prove an extension of Theorem \ref{mainthm3}. More precisely, we prove the following.

\begin{thm}\label{mainthm4}
Let $(M,K)$ be an m-th root Cartan space. Then the following are
equivalent:
\begin{description}
\item[a)]$K$ has isotropic mean Berwald curvature, i.e., ${\bf E}=\frac{n+1}{2} c K{\bf h}$;
\item[b)]$K$ has vanishing $E$-curvature, i.e., ${\bf E}=0$;
\item[c)]$K$ has almost isotropic $S$-curvature, i.e., ${\bf S}=(n+1)cK+\eta$;
\end{description}
where $c=c(x)$ is a scalar function and $\eta=\eta_i(x)y^i$ is a
1-form on $M$.
\end{thm}

\bigskip

To prove the Theorem \ref{mainthm4}, first we show the following.
\begin{lem}\label{lemB}
Let $(M,K)$ be an $n$-dimensional $m$-th root Cartan space. Then the
following are equivalent:
\begin{description}
\item[a)]${\bf S}=(n+1)c K+\eta;$
\item[b)]${\bf S}=\eta;$
\end{description}
where $c=c(x)$ is a scalar function and $\eta=\eta^i(x)p_i$ is a 1-form on $M$.
\end{lem}
\begin{proof}
By lemma \ref{lemG}, the $E$-curvature of an m-th root metric is a
rational function . On the other hand, by taking twice vertical
covariant derivatives of the $\bf S$-curvature,we get the
$E$-curvature. Thus $\bf S$-curvature is a rational function.
Suppose that $K$ has almost isotropic $\bf S$-curvature, ${\bf
S}=(n+1)c K+\eta$, where $c=c(x)$ is a scalar function and
$\eta=\eta^i(x)p_i$ is a 1-form on $M$. Then the left hand side of
${\bf S}-\eta=(n+1)c(x) K$ is rational function while the right hand
is irrational function. Thus $c=0$ and ${\bf S}=\eta$.
\end{proof}

\begin{lem}\label{lemC}
Let $(M,K)$ be an $n$-dimensional $m$-th root Cartan space. Then the
following are equivalent:
\begin{description}
\item[a)]${\bf E}=\frac{n+1}{2}cK{\bf h};$
\item[b)]${\bf E}=0;$
\end{description}
where $c=c(x)$ is a scalar function.
\end{lem}

\begin{proof}
Suppose that $K$ has isotropic mean Berwald curvature
\begin{eqnarray}
{\bf E}=\frac{n+1}{2}c K {\bf h},\label{m11}
\end{eqnarray}
where $c=c(x)$ is a scalar function. The left hand side of
(\ref{m11}), is a rational function while the right hand is
irrational function. Thus $c=0$ and $E=0$.
\end{proof}

\bigskip

\noindent
{\it\bf Proof of Theorem \ref{mainthm4}:}  By Lemmas \ref{lemB} and \ref{lemC}, we get the proof.

\bigskip

\begin{cor}
Let $(M,K)$ be an $n$-dimensional $m$-th root Cartan space. Suppose
that $K$ has isotropic $\bf S$-curvature, ${\bf S}=(n+1)c(x) K$, for
some scalar function $c=c(x)$ on $M$. Then ${\bf S}=0$.
\end{cor}

%------------------------------------------------------------------------------------------------------------
\section{ Proof of Theorem \ref{mainthm2} }
%------------------------------------------------------------------------------------------------------------

\noindent {\it\bf Proof of Theorem \ref{mainthm2}}: Let $(M,K)$ be
an $n$-dimensional $m$-th root Cartan space. Suppose that $K$ be of almost
vanishing $\bf H$-curvature, i.e.,
\begin{eqnarray}
H^{ij}=\frac{n+1}{2K}\theta h^{ij},\label{m7}
\end{eqnarray}
where $\theta$ is a 1-form on $M$. The angular metric
$h^{ij}=g^{ij}-K^2p^ip^j$ is given bye the following
\begin{eqnarray}
&&h^{ij}=(m-1)(a^{ij}-a^ia^j),\label{m8}
\end{eqnarray}
Plugging (\ref{m8}) into (\ref{m7}) yields
\begin{eqnarray}
H^{ij}=\frac{n+1}{2K}\theta[(m-1)(a^{ij}-a^ia^j)].\label{m9}
\end{eqnarray}
By Lemma \ref{lemG} and $H^{ij}=E^{ij|s}p_s$, it is easy to see that
$H^{ij}$ is rational with respect to $y$. Thus, (\ref{m9}) implies that $\theta=0$ or
\begin{eqnarray}
(m-1)(a^{ij}-a^ia^j)=0.\label{m10}
\end{eqnarray}
By (\ref{m8}) and (\ref{m10}), we conclude that $h^{ij}=0$, which is
impossible. Hence $\theta=0$ and  then $H^{ij}=0$.
\qed

\bigskip
\noindent
Akbar Tayebi and Ali Nankali\\
Department of Mathematics, Faculty  of Science\\
University of Qom \\
Qom. Iran\\
Email:\ akbar.tayebi@gmail.com\\
Email:\ ali.nankali2327@yahoo.com
\bigskip

\noindent
Esmaeil Peyghan\\
Department of Mathematics, Faculty  of Science\\
Arak University\\
Arak 38156-8-8349,  Iran\\
Email: epeyghan@gmail.com

\end{document}